\documentclass[12pt]{amsart}
\usepackage{graphicx}
\usepackage{float}
\usepackage{amsmath,amsthm, amsfonts,amssymb, stmaryrd,yfonts,pxfonts,pifont,eufrak, bbm}
\usepackage{tikz}
\usepackage{caption}
\usepackage{subcaption}

\newcommand{\N}{\mathbb{N}}

\newcommand{\Z}{\mathbb{Z}}

\newtheorem{corollary}{Corollary}

\newtheorem{result}{Proposition}
\newtheorem{example}{Example}
\newtheorem{Remark}{Remark}

\theoremstyle{definition}

\newtheorem{remark}{Remark}
\numberwithin{equation}{subsection}

\begin{document}
	\title[On Equicontinuity and Related Notions in Nonautonomous Dynamical Systems]{On Equicontinuity and Related Notions in Nonautonomous Dynamical Systems}%
	\author{Sushmita Yadav and Puneet Sharma}
	\address{Department of Mathematics, I.I.T. Jodhpur, NH 65, Nagaur Road, Karwar, Jodhpur-342030, INDIA}%
	\email{yadav.34@iitj.ac.in, puneet@iitj.ac.in}%
	
	
	

	
	\maketitle
\section*{Abstract}
In this work, we investigate the dynamics of a general non-autonomous system generated by a commutative family of homeomorphisms. In particular, we investigate properties such as periodicity, equicontinuity, minimality and transitivity for a general non-autonomous dynamical system. In \cite{sk2}, the authors derive necessary and sufficient conditions for a system to be minimal. We claim the result to be false and provide an example in support of our claim. Further, we correct the result to derive necessary and sufficient conditions for a non-autonomous system to be minimal. We prove that for an equicontinuous flow generated by a commutative family, while the system need not exhibit almost periodic points, if $x$ is almost periodic then every point in $\overline{\mathcal{O}_H(x)}$ is almost periodic. We further prove that in such a case, the set $\overline{\mathcal{O}_H(x)}$ is uniformly almost periodic and hence provide an analogous extension to a result known for the autonomous systems. We prove that a system generated by a commutative family is transitive if and only if it exhibits a point with dense orbit. We also prove that any minimal system generated by commutative family is either equicontinuous or has a dense set of sensitive points.

\section*{Introduction}

Dynamical systems have been long used to investigate various natural and physical processes around us. While mathematical investigations have enriched the literature with qualitative results determining long term behavior of such systems, the field has also found a variety of applications in areas such as complex systems, control theory, biomechanics and cognitive sciences \cite{beer,Hamill,kohmoto}.  Although the theory has been used extensively in various fields, most of problems have been approximated using autonomous systems (systems with time invariant governing rule). However, as governing rule varies with time in many natual processes around us, better approximations can be obtained by allowing the governing rule to be time variant. While some investigations for such a setting in the discrete case have been made, many questions for such a setting are still unanswered. In \cite{sk1}, the authors investigate the topological entropy when the family $\mathbb{F}$ is equicontinuous or uniformly convergent. In \cite{sk2}, ther authors discusses minimality conditions for a non-autonomous system on a compact Hausdorff space while focussing on the case when the non-autonomous system is defined on a compact interval of the real line. In \cite{pm}, the authors investigate a non-autonomous system generated by a finite family of continuous self maps. In the process, they study properties such as transitivity, weak mixing, topological mixing, existence of periodic points, various forms of sensitivities and Li-Yorke chaos. In \cite{jd}, the authors establish that if $f_n\rightarrow f$, there is no relation between chaotic behavior of the non-autonomous system generated by $f_n$ and the chaotic behavior of $f$. Before we move further, we give some of the basic concepts and definitions required.\\

Let $(X,d)$ be a compact metric space and let $\mathbb{F}=\{f_n: n \in \mathbb{N}\}$ be a family of homeomorphisms on $X$. For any given initial state of the system $x_0$, any such family generates a \emph{non-autonomous} dynamical system via the relation $x_{n} = \left\{
	\begin{array}{lr}
		{f_n(x_{n-1})} &  : n\geq 1, \\
		{f_n^{-1}(x_{n+1})} & : n<0. \\
	\end{array}
	\right.$ In other words, the non-autonomous system generated by the family $\mathbb{F}$ can be visualized as orbit of $x_0$ under the ordered set $\{\ldots,f_2^{-1},f_1^{-1},I,f_1,f_2,\ldots,\}$. For a given initial state $x_0$ of the system, let $\omega_n(x_0)$ denote the state of the system at time $n$. The set $\mathcal{O}(x)=\{\omega_n(x): n\in \mathbb{Z}\}$ is called the \emph{orbit} of any point $x$ in $X$. Further, we refer to the set $\mathcal{O}_H(x)=\{(\omega_{k_n}\circ\omega_{k_{n-1}}\circ\ldots\circ\omega_{k_1})(x): k_i\in \Z, n\in \N\}$ as the orbital hull of the point $x$. Let $\mathcal{O}_H^k(x)=\{\omega_{r_n}\omega_{r_{n-1}}...\omega_{r_2}\omega_{r_1}(x) : n\in \N, r_i\in\{-k,-k+1,\ldots,1,2,...,k\}\}$ denote the truncation (of order $k$) of the orbitall hull of the point $x$. It may be noted that orbital hull of a point $x$ is the smallest invariant set containing $x$. \\

A point $x\in X$ is said to be \emph{periodic} of period $n\in \N$ if $\omega_{nk}(x)=x$ $\forall k \in \Z$.  A point $x\in X$ is called \emph{almost periodic} if for any $\epsilon>0$, the set $\{n\in \Z: d(\omega_n(x),x)<\epsilon\}$ is syndetic. If every point $x\in X$ is \emph{almost periodic} then $(X,\mathbb{F})$ is said to be \emph{pointwise almost periodic}. $(X,\mathbb{F})$ is said to be \emph{uniformly almost periodic} if for every $\epsilon>0$ there exists $M>0$ such that the set $\{n\in Z : d(\omega_n(x),x)<\epsilon \}$ is $M$-syndetic for all $x\in X$. A set $Y\subseteq X$ is said to be \emph{invariant} if $\omega_k(Y)\subseteq Y$, for all $k\in \Z$. We say $(Y,\mathbb{F})$ to be a \emph{minimal subsystem} of $(X,\mathbb{F})$ if it is a non-empty, closed, invariant subsystem of $(X,\mathbb{F})$ with no proper non-empty subset having these properties. A system $(X,\mathbb{F})$ is said to be \emph{equicontinuous} if for each $\epsilon>0$, there exists $\delta>0$ such that $d(x, y) < \delta$ implies $d(\omega_n(x), \omega_n(y)) < \epsilon$ for all $n\in \Z$, $x, y \in X$. A pair $(x, y)$ is \emph{proximal} for $(X,\mathbb{F})$  if $\liminf\limits_n ~~d(\omega_n(x), \omega_n(y)) = 0$. Let $P(X)$ denote the set of proximal pairs of system $(X,\mathbb{F})$, then $(X,\mathbb{F})$ is said to be \emph{distal} if $P(X)=\Delta$, where $\Delta$ denotes the diagonal in the space $X\times X$. A system $(X,\mathbb{F})$ is said to be \emph{point transitive} if there exists a point $x\in X$ such that $\overline{\mathcal{O}(x)}=X$. In this case, the point $x$ is referred as a \emph{transitive point}. The system is said to be \emph{topologically transitive} if for every pair of non-empty open sets $U, V$ in $X$, there exists $k\in \Z$ such that $\omega_k(U)\cap V \neq \emptyset$. The system $(X,\mathbb{F})$ is called $r$-transitive if the system generated by the family $\mathbb{F}_r=\{f_{(k+1)r}\circ f_{(k+1)r-1}\circ\ldots\circ f_{kr+1}:k\in\N\}$ is transitive. A system $(X,\mathbb{F})$ is totally transitive if it is $r$-transitive for all $r\in\mathbb{N}$. A system $(X,\mathbb{F})$ is said to be \emph{sensitive} at a point $x$ if there exists $\delta_x>0$ such that for each neighborhood $U_x$ of $x$ there exists $k\in\Z$ such that $diam(\omega_k(U_x))>\delta_x$. A system $(X,\mathbb{F})$ is said to be \emph{sensitive} if there exists $\delta>0$ such that for each $x\in X$ and each neighborhood $U$ of $x$ there exists $k\in\Z$ such that $diam(\omega_k(U))>\delta$. It may be noted that in case the $f_n$'s coincide, the above definitions coincide with the known notions of an autonomous dynamical system \cite{bc,bs,de}. Some basic concepts and recent works in this area can be found in literature \cite{kol,of,jd,sk1,sk2,pm}.\\

In this paper, we investigate properties such as periodicity, equicontinuity, minimality and transitivity for a non-autonomous system generated by a commutative family of homeomorphisms. We prove that every point in the orbital hull of periodic point is periodic. We give example to show that a pair of periodic points may form a Li-Yorke pair and hence need not exhibit simple dynamical behavior. In \cite{sk2}, the authors claim that  system is minimal if and only if orbit of each point is dense in $X$ (page $84$, line $-7$). Also, the authors claim that a non-autonomous system is minimal if and only if for non-empty every open set $U$ in $X$, there exists $k\in\N$ such that trajectory of every point meets $U$ in at most $k$ iterations (Lemma $2.2$).  We establish both the claims to be false. While we provide an example of a minimal system void of any points with dense orbit, we correct the result to derive necessary and sufficient conditions for a non-autonomous system to be minimal. We prove that a non-autonomous system is minimal if and only if for non-empty every open set $U$ in $X$, there exists $k\in\N$ such that $\mathcal{O}_H^k(x)$ meets $U$ for every point $x\in X$. We prove that for equicontinuous systems, if $x$ is almost periodic then every point in $\overline{\mathcal{O}_H(x)}$ is almost periodic. In such a setting, we establish $\overline{\mathcal{O}_H(x)}$ to be uniformly almost periodic. We prove that a system is transitive if and only if it exhibits points with dense orbit. We prove that while minimal systems need not be transitive, an equicontinuous transitive system is necessarily minimal. We also prove that a minimal system is either equicontinuous or has a dense set of points of sensitivity.

\section*{Main Results}

\begin{result}
For any system $(X,\mathbb{F})$ generated by a commutative family of homeomorphisms, $x$ is periodic for $(X,\mathbb{F}) \implies$ each point of $\overline{\mathcal{O}_H(x)}$ is periodic (with same period).
\end{result}
	
\begin{proof}
Let $(X,\mathbb{F})$ be generated by a commutative family of homeomorphisms and let $x$ be periodic for $(X,\mathbb{F})$ (with period r). Then, $\omega_{nr}(x)=x~~$ for all $n\in\mathbb{Z}$. For any point $\omega_{r_n}\circ\omega_{r_{n-1}}\circ\ldots\omega_{r_1}(x)$ in orbital hull of $x$, as $\omega_{nr}(\omega_{r_n}\circ\omega_{r_{n-1}}\circ\ldots\omega_{r_1}(x))=\omega_{r_n}\circ\omega_{r_{n-1}}\circ\ldots\omega_{r_1}(\omega_{nr}(x))=\omega_{r_n}\circ\omega_{r_{n-1}}\circ\ldots\omega_{r_1}(x)$ (as $\mathbb{F}$ is commutative), every point in $\mathcal{O}_H(x)$ is periodic (with same period). Finally, as the limit of periodic points of period $k$ is a periodic point of period $k$, each point of $\overline{\mathcal{O}_H(x)}$ is periodic point of period $k$ and the proof is complete. 	
\end{proof}

\begin{Remark}
The above result establishes the periodicity of the elements of the closure of the orbital hull of a point $x$, when the point $x$ is itself periodic. It may be noted that as the governing rule is time variant, the periodicity of $x$ need not guarantee the periodicity of the members of the orbital hull (or even orbit itself). Also, if the governing rule is time variant, a periodic point may have infinite orbit and hence need not attribute to simpler dynamical behavior. In fact, if the governing rule is time variant, a pair of periodic points may behave in an unexpected manner and form a Li-Yorke pair. We now give examples in support of our claim.
\end{Remark}
	
\begin{example}
Let $X=[0,1]$ be the unit interval and let $f_n:X\rightarrow X$ be defined as $f_{2n-1}(x) = \left\{
		\begin{array}{lr}
			\frac{x}{2} &  : 0\leq x \leq \frac{1}{2}, \\
			{\frac{3}{2}x-\frac{1}{2}} & : \frac{1}{2}\leq x\leq 1 \\
		\end{array}
		\right\}. ~~~    where~ k\in \N$ and $f_{2n}(x)=1-\sqrt{x}$ for all $n\in\mathbb{N}$.
		
Then, $\frac{1}{2}\in X$ is a periodic point of period 2, but $f_1(\frac{1}{2})=\frac{1}{4}$ is not periodic for $(X,\mathbb{F})$. Thus, periodicity need not be preserved in the elements of the orbital hull when the generating maps do not commute.
\end{example}

\begin{example}
Let $X=[0,1]$ and define $f_n:X\rightarrow X$ such that $f_{2n-1}(x)=x^{2n}$ and $f_{2n}(x)=x^{\frac{1}{2n}}$ for $n\in \N$. Then, every point is periodic (with period $2$). Also as $x^n$ converges to $0$ for all $x$ in $(0,1)$, the pair $(x,y)$ forms a Li-Yorke pair for all $x,y\in (0,1)$. Consequently, periodic points may form a Li-Yorke pair and hence need not attribute to simpler dynamical behavior.
\end{example}
	
\begin{result}
For any system $(X,\mathbb{F})$, $(X,\mathbb{F})$ is minimal if and only if $\overline{\mathcal{O}_H(x)}=X$ for all $x\in X$.
\end{result}
	
\begin{proof}
As orbital hull of $x$ is smallest invariant subset of $X$ containing $x$, $(X,\mathbb{F})$ is minimal if and only if $\overline{\mathcal{O}_H(x)}=X$ for all $x\in X$.
\end{proof}
	
\begin{Remark}
The above result provides necessary and sufficient criteria for a non-autonomous system to be minimal. It may be noted that $X$ is itself invariant for the system generated by the family $\mathbb{F}$, a simple application of Zorn's lemma establishes the existence of minimal sets for non-autonomous systems. In \cite{sk2}, the authors claim that  system is minimal if and only if orbit of each point is dense in $X$ (page $84$, line $-7$). Further, the authors claim that a non-autonomous system is minimal if and only if for non-empty every open set $U$ in $X$, there exists $k\in\N$ such that trajectory of every point meets $U$ in at most $k$ iterations (Lemma $2.2$).  However, we claim that both of the observations fail to hold. It may be noted that as minimality in non-autonomous systems is equivalent to orbital hull of each point being dense in $X$, minimality of a system does not guarantee orbit of each point to be dense in $X$ (Example \ref{Example 2}). Also, as orbit of a point is a non-invariant set, the second assertion also fails to hold good (Example \ref{Example 2}).  In fact, as orbital hull of $x$ is the smallest invariant set containing $x$, a non-autonomous system is minimal if and only if for non-empty every open set $U$ in $X$, there exists $k\in\N$ such that $\mathcal{O}_H^k(x)$ meets $U$ for every point $x\in X$ (Proposition \ref{mo}). Also, while minimality of a system ensures each of its points to be almost periodic in the autonomous case, non-invariance of the governing rule forces such an implication not to hold true for non-autonomous systems. The proof follows from the fact that if the governing rule varies with time, any initial point $x_0$ may fail to return to its neighborhood even in the absence of proper invariant sets. In fact contrary to the autonomous case, a minimal set in a non-autonomous system may contain periodic points in the non-trivial sense (periodic points whose orbit is proper in the minimal set).  We now give examples in support of our claim.
\end{Remark}
	
\begin{example}\label{Example 2}
		Let $X=\mathbb{S}^1$ be the unit circle and let $f_1(\theta)=\theta+\frac{1}{2}$, $f_2(\theta)=\theta-\frac{1}{2^2}$. For $n\geq 3$, define $f_n(\theta) = \left\{
		\begin{array}{lr}
			{\theta+\frac{1}{2^{k}}} &  :n=2k+1, \\
			{\theta-\frac{1}{2^k}-\frac{1}{2^{k+1}}} & : n=2k. \\
		\end{array}
		\right. ~~~    where~ k\in \N$
		
		As closure of the orbital hull of each $x$ is $X$, the non-autonomous system generated by the family $(f_n)$ is minimal. However as each point settles at the diametrically opposite end (of the initial point $x_0$), none of the points in the system are almost periodic. Consequently, almost periodic points are not guaranteed to exist in the non-autonomous setting.
\end{example}

\begin{example}\label{ex4}
	Define $f_n: \mathbb{S}^1\rightarrow \mathbb{S}^1$ as follows:
	$$f_n(\theta) = \left\{
	\begin{array}{lr}
		{\theta+\frac{1}{2^{k}}} &  :n=4k ~or~ 4k-3, \\
		{\theta-\frac{1}{2^k}} & : n=4k-1 ~or~ 4k-2. \\		
	\end{array}
	\right. ~~~    where~ k\in \N$$
	It is clear that every element in $\mathbb{S}^1$ is periodic with period 2. However as orbital hull of every point in dense in $S^1$, the system $(X,\mathbb{F})$ is minimal and contains periodic points in the non-trivial sense.
\end{example}

\begin{result}\label{mo}
Any system $(X,\mathbb{F})$ is minimal if and only if for every non-empty open set $U$ in $X$ there exists $k \in \N$ such that set $\mathcal{O}_H^k(x)\cap U\neq \emptyset$ for all $x\in X$.
\end{result}

\begin{proof}
Let $(X,\mathbb{F})$ be minimal and let $U$ be a non-empty open set in $X$. Firstly, note that if the claim does not hold then for each $k\in\mathbb{N}$ there exists $x_k$ such that $\mathcal{O}_H^k(x_k)\cap U=\emptyset$. Let $x$ be a limit point of $(x_k)$ and let $(r_1,r_2,\ldots,r_m)$ be any tuple (of any fixed length $m$). Let $p=\max \{r_i:i=1,2,\ldots,m\}$. As $\mathcal{O}_H^k(x_k)\cap U=\emptyset$, we have $(\omega_{r_m}\circ\omega_{r_{m-1}}\circ\ldots\circ\omega_{r_1})(x_n)\notin U$ for all $n\geq p$ and hence $\omega_{r_m}\circ\omega_{r_{m-1}}\circ\ldots\circ\omega_{r_1}(x)\notin U$. As the argument holds for any tuple $(r_1,r_2,\ldots,r_m)$, we have $\mathcal{O}_H(x)\cap U=\emptyset$ (which contradicts minimality of $X$). Consequently, there exists $k \in \N$ such that set $\mathcal{O}_H^k(x)\cap U\neq \emptyset$ for all $x\in X$ and the proof of forward part is complete.\\

Conversely, if there exists $k \in \N$ such that set $\mathcal{O}_H^k(x)\cap U\neq \emptyset$ for all $x\in X$, then orbital hull of any point $x$ intersects every non-empty open set $U$ and hence $X$ is minimal.
\end{proof}

\begin{result}
For any system $(X,\mathbb{F})$ generated by a commutative family of homeomorphisms, $(X,\mathbb{F})$ is equicontinuous $\implies$ $(X,\mathbb{F})$ is distal.
\end{result}
	
\begin{proof}
Let $(X,\mathbb{F})$ be equicontinuous and $x$ and $y$ be proximal. Then thee exists sequence $(n_k)$ of integers such that $\lim \limits_{k\rightarrow \infty} d(\omega_{n_k}(x), \omega_{n_k}(y))=0$. For any $\epsilon>0$, there exists $\delta>0$ such that $d(a,b)<\delta \implies d(\omega_n(a),\omega_n(b))<\epsilon ~~\forall n \in \mathbb{Z}$. As $d(\omega_{n_r}(x), \omega_{n_r}(y))<\delta$ (for some $n_r$), we have $d(x,y)<\epsilon$. As the argument holds for any $\epsilon>0$, we have $x=y$ and hence the system $(X,\mathbb{F})$ is distal.
\end{proof}

\begin{result}
For any non-autonomous system $(X,\mathbb{F})$ generated by a commutative family of homeomorphisms, if $x$ is almost periodic then every point in $\mathcal{O}(x)$ is almost periodic.
\end{result}

\begin{proof}
Let $(X,\mathbb{F})$ be generated by a commutative family of homeomorphisms and let $x$ be an almost periodic point. Let $k\in \mathbb{Z}$ and let $U$ be a neighborhood of $\omega_k(x)$. Then as $x$ is almost periodic, the set $\{r\in\mathbb{Z}:\omega_r(x)\in \omega_k^{-1}(U)\}$ is syndetic and hence the set $\{r\in\mathbb{Z}:(\omega_k\circ\omega_r)(x)\in U\}$ is syndetic. Consequently, the set $\{r\in\mathbb{Z}:\omega_r(\omega_k)(x)\in U\}$ is syndetic and thus $\omega_k(x)$ is almost periodic. As the argument holds for any $k$, every point in the orbit of $x$ is almost periodic.
\end{proof}

\begin{Remark}\label{ap}
The above remark establishes the almost periodicity of elements in the orbit when the initial point $x$ is almost periodic. The proof uses the fact that if the initial point $x$ is almost periodic then the syndetic bound of some neighborhood of $x$ carries forward to the neighborhood of the given point in the orbit. As the arguments establish the almost periodicity of elements in the orbit of $\omega_k(x)$, a similar argument ensures almost periodicity of elements in the orbit of $\omega_k(x)$. Thus,  almost periodicity of $x$ ensures almost periodicity of elements of $\mathcal{O}_H(x)$. Further for any equicontinuous system, a similar argument establishes almost periodicity of any limit of almost periodic points and hence we get the following result.
\end{Remark}

\begin{result}
For any non-autonomous system $(X,\mathbb{F})$ generated by a commutative family of homeomorphisms, if $x$ is almost periodic then every point in $\mathcal{O}_H(x)$ is almost periodic. Further, if $(X,\mathbb{F})$ is equicontinuous then every point in $\overline{\mathcal{O}_H(x)}$ is almost periodic.
\end{result}

\begin{proof}
Let $(X,\mathbb{F})$ be a non-autonomous system and let $x\in X$. As almost periodicity of $x$ guarantees $\omega_k(x)$ to be almost periodic for all $k\in\mathbb{Z}$ (Proposition \ref{ap}), every point in $\mathcal{O}_H(x)$ is almost periodic. Further, let $(X,\mathbb{F})$ be equicontinuous, $y\in\overline{\mathcal{O}_H(x)}\setminus\mathcal{O}_H(x)$ be fixed and $\epsilon>0$ be given. As $(X,\mathbb{F})$ is equicontinuous, there exists $\delta<\frac{\epsilon}{3}$ such that $d(x,y)<\delta$ $\implies$ $d(\omega_k(x),\omega_k(y))<\frac{\epsilon}{3}$ for all $k\in\mathbb{Z}$. Also as $y\in\overline{\mathcal{O}_H(x)}\setminus\mathcal{O}_H(x)$, there exists $p(x)=\omega_{k_t}\omega_{k_{t-1}}...\omega_{k_2}\omega_{k_1}(x)\in \mathcal{O}_H(x)$ such that $d(p(x),y)<\delta$. As $p(x)$ is almost periodic, the set $\{k_r : d(\omega_{k_r}(p(x)), p(x))<\frac{\epsilon}{3}\}$ is syndetic. Further, as $d(\omega_{k_r}(y),y)\leq d(\omega_{k_r}(y),\omega_{k_r}(p(x)))+d(\omega_{k_r}(p(x)),p(x))+d(p(x),y)<\epsilon$ the set $\{m : d(\omega_{m}(y), y)<\epsilon\}$ is syndetic. Consequently every point in $\overline{\mathcal{O}_H(x)}$ is almost periodic and the proof is complete.
\end{proof}

\begin{result}
For any equicontinuous system $(X,\mathbb{F})$ generated by a commutative family of homeomorphisms, if $x\in X$ is an almost periodic point then $\overline{\mathcal{O}_H(x)}$ is uniformly almost periodic.
\end{result}
	
\begin{proof}
Let $\epsilon>0$ and $u\in \overline{\mathcal{O}_H(x)}$ be any arbitrary element. As $(X,\mathbb{F})$ is equicontinuous, there exists a $\delta>0$ such that $d(x,y)<\delta\implies d(\omega_k(x),\omega_k(y))<\frac{\epsilon}{4},~ \forall x,y\in X, k\in \Z$. Let $\eta=min\{\frac{\epsilon}{4},\delta\}$ and let $F=\{x_1,x_2,...x_n\}$ be a $\eta$-dense set in $\overline{\mathcal{O}_H(x)}$. As $F$ is $\eta$-dense, there exists $x_r\in F$ such that $d(u,x_r)<\eta$ and hence $d(\omega_k(x_r),\omega_k(u))<\frac{\epsilon}{4}$ for all $k\in \Z$. Consequently, if orbit of $x_r$ returns to its $\frac{\epsilon}{4}$-neighborhood syndetically (at times $(n_r)$), $d(u,\omega_{n_r}(u))\leq d(u,x_i)+d(x_i,\omega_{n_r}(x_i))+d(\omega_{n_r}(x_i),\omega_{n_r}(u))<\epsilon$ and hence orbit of $u$ returns to its $\epsilon$-neighborhood syndetically (at same set of times $(n_r)$). As $(X,\mathbb{F})$ is equicontinuous, there exists a common syndetic set for $\{x_1,x_2,\ldots,x_r\}$ and hence every point returns to its $\epsilon$-neighborhood sydetically with the same syndetic set. Thus, $\overline{\mathcal{O}_H(x)}$ is uniformly almost periodic and the proof is complete.
\end{proof}


\begin{result}
For any non-autonomous system $(X,\mathbb{F})$ generated by a commutative family of homeomorphisms, if $(X,\mathbb{F})$ is equicontinuous then $\overline{\mathcal{O}_H(x)}=\overline{\mathcal{O}_H(y)}$ for all $y\in\overline{\mathcal{O}(x)}$.
\end{result}

\begin{proof}
Let $(X,\mathbb{F})$ be an equicontinuous system generated by a commutative family of homeomorphisms and let $x\in X$.  Let $y\in \overline{\mathcal{O}(x)}$ and $\epsilon>0$ be given. As $(X,\mathbb{F})$ is equicontinuous, there exists $\delta>0$ such that $d(a,b)<\delta$ ensures $d(\omega_{n}(x),\omega_{n}(y))<\epsilon~~ \forall n\in \Z$. Also $y\in \overline{\mathcal{O}(x)}$ forces some $n_k\in\Z$ such that $d(\omega_{n_k}(x),y)<\delta$ and hence $d(x,\omega_{-n_k}(y))<\epsilon$. As the argument holds for any $\epsilon>0$, we have $\overline{\mathcal{O}(x)}= \overline{\mathcal{O}(y)}$. Further, as $x\in \overline{\mathcal{O}(y)}\subset \overline{\mathcal{O}_H(y)}$ implies $\overline{\mathcal{O}_H(x)}\subset \overline{\mathcal{O}_H(y)}$ (as orbital hull of $x$ is the smallest invariant set containing $x$), $x$ and $y$ have identical orbital hulls and hence elements in the orbit closures generate orbital hulls identical to the original point and the proof is complete.
\end{proof}

\begin{remark}\label{oh}
The above result establishes that if the system $(X,\mathbb{F})$ is equicontinuous, $\overline{\mathcal{O}_H(y)}=\overline{\mathcal{O}_H(x)}$ for any point $y$ in $\overline{\mathcal{O}(x)}$. As the arguments can repeated for elements of the orbit, $\overline{\mathcal{O}_H(y)}=\overline{\mathcal{O}_H(x)}$ for any point $y$ in $\mathcal{O}_H(x)$. It may be noted that if $\{(\omega_{k_n}\circ\omega_{k_{n-1}}\circ\ldots\circ\omega_{k_1}): k_i\in \Z, n\in \N\}$ is an equicontinuous family, then similar arguments establish $\overline{\mathcal{O}_H(y)}=\overline{\mathcal{O}_H(x)}$ for any point $y$ in $\overline{\mathcal{O}_H(x)}$. Consequently, all elements in $\overline{\mathcal{O}_H(x)}$ generate the same set (equal to $\overline{\mathcal{O}_H(x)}$) and hence the system is minimal. However, if the family under discussion fails to be equicontinuous, the non-autonomous system may fail to be minimal (even when the generating family is commutative). Thus we get the following results.\\
\end{remark}

\begin{corollary}
For any non-autonomous system $(X,\mathbb{F})$ generated by a commutative family of homeomorphisms, if $\{(\omega_{k_n}\circ\omega_{k_{n-1}}\circ\ldots\circ\omega_{k_1}): k_i\in \Z, n\in \N\}$ is equicontinuous then every point in $X$ generates a minimal subsystem of $(X,\mathbb{F})$.
\end{corollary}

\begin{proof}
The proof follows from discussions in Remark \ref{oh}.
\end{proof}

\begin{example}
Let $X=[0,1]$ and define $f_n:X\rightarrow X$ such that $f_{2n-1}(x)=x^{2}$ and $f_{2n}(x)=\sqrt{x}$ for $n\in \N$. Then as every point is periodic  (with finite orbit), the system $(X,\mathbb{F})$ is equicontinuous. However, as $0$ is a fixed point such that $0\in \overline{\mathcal{O}_H(\frac{1}{2})}$, we have $\overline{\mathcal{O}_H(0)}\neq \overline{\mathcal{O}_H(\frac{1}{2})}$ and the above result cannot be generalized to elements in $\overline{\mathcal{O}_H(x)}$.
\end{example}
		
\begin{result}\label{st1}
For any non-autonomous dynamical system $(X,\mathbb{F})$ generated by a commutative family $\mathbb{F}$, $(X,\mathbb{F})$ is transitive if and only if it has a point with dense orbit.
\end{result}

\begin{proof}
Let $(X,\mathbb{F})$ be a transitive system such that no point in $X$ has dense orbit in $(X,\mathbb{F})$. As $X$ is compact, for each $k\in\mathbb{N}$ there exists a finite subset $F_k$ such that $F_k$ is $\frac{1}{k}$-dense in $X$. As no point in $X$ has dense orbit, for each $x\in X$ there exists $r\in\mathbb{N}$ and $x_r\in F_r$ such that $\mathcal{O}(x)\cap S(x_r,\frac{1}{r})=\emptyset$. However, as $(X,\mathbb{F})$ is transitive, $O_{x,r}=\bigcup\limits_{n=0}^{\infty} \omega^{-1}_n (S(x_r,\frac{1}{r}))$ is an open dense subset of $X$. Thus $C_{x,r}=O_{x,r}^c$ is a non-empty closed nowhere dense subset of $X$. As choices of $x_r$ are countable, $X$ is  countable union nowhere dense subsets of $X$ which contradicts compactness of $X$ and hence $X$ must have a point with dense orbit.\\

Conversely, let $x\in X$ such that $\mathcal{O}(x)$ is dense in $X$ and let $U$ and $V$ be non-empty open subsets of $X$. As orbit of $x$ is dense, there exists $k\in\mathbb{N}$ such that $\omega_k(x)\in U$. Further, note that $\mathcal{O}(\omega_k(x))=\{\omega_n(\omega_k(x)):n\in\mathbb{N}\} =\{\omega_k(\omega_n(x)):n\in\mathbb{N}\} =\omega_k(\mathcal{O}(x))$ is dense in $X$ (as $\mathbb{F}$ is commutative and surjective). Thus there exists $r\in\mathbb{N}$ such that $\omega_r(\omega_k(x))\in V$ and hence $\omega_r(U)\cap V\neq\emptyset$. Consequently, $(X,\mathbb{F})$ is transitive and the proof is complete.
\end{proof}

\begin{remark}
The above proof establishes a necessary and sufficient criteria to establish transitivity of a non-autonomous dynamical system $(X,\mathbb{F})$. In particular, the result establishes existence of a dense orbit to be an equivalent criteria for a non-autonomous system to be transitivity. In addition, if the system is equicontinuous then the points in the space move in a synchronized manner (which may be visualized better via uniform almost periodicity) and hence denseness of an orbit (in $X$) forces denseness of all the orbits (in $X$). Hence we get the following result.
\end{remark}

\begin{result}
For any equicontinuous non-autonomous system $(X,\mathbb{F})$ generated by a commutative family of homeomorphisms, $(X,\mathbb{F})$ is transitive if and only if $\mathcal{O}(x)$ is dense in $X$ for all $x\in X$ (and hence $(X,\mathbb{F})$ is minimal). 
\end{result}
	
\begin{proof}
Let $(X,\mathbb{F})$ be an equicontinuous system generated by a commutative family of homeomorphisms and let $(X,\mathbb{F})$ be transitive. By Proposition \ref{st1}, there exists $x\in X$ such that $\mathcal{O}(x)$ is dense in $X$. Let $y\in X$ be arbitrary and let $U=S(z,\epsilon)$ be any non-empty open subset of $X$. As $(X,\mathbb{F})$ is equicontinuous, there exists $\delta>0$ such that $d(a,b)<\delta$ ensures $d(\omega_k(a),\omega_k(b))<\frac{\epsilon}{4}$ for all $a,b\in X~~, k\in \Z$. As $\mathcal{O}(x)$ is dense in $X$, there exists $u\in \mathcal{O}(x)$ such that $d(u,y)<\delta$ and hence  $d(\omega_k(u),\omega_k(y))<\frac{\epsilon}{4}$ for all $k\in\Z$. As denseness of $\mathcal{O}(x)$ (in $X$) forces $\overline{\mathcal{O}(u)}=X$, we have $d(z,\omega_r(u))<\frac{\epsilon}{4}$ for some $r\in \Z$. Thus, $d(z,\omega_r(y))<d(z,\omega_r(u))+d(\omega_r(u),\omega_r(y))<\epsilon$ and hence orbit of $y$ intersects $S(z,\epsilon)$. As the argument holds for any $y\in X$, orbit of any point is dense in $X$. As the proof of converse is trivial, $\mathcal{O}(x)$ is dense in $X$ for some $x\in X$ if and only if $\mathcal{O}(x)$ is dense in $X$ for all $x\in X$. Finally, as denseness of orbit of a point forces denseness of orbital hull (of the same point), $(X,\mathbb{F})$ is minimal and the proof is complete.
\end{proof}

\begin{remark}\label{tth}
The above proof establishes a necessary and sufficient criteria for an equicontinuous system to be transitive. It may be noted that as minimality of a system does not guarantee transitivity in the non-autonomous case, a minimal equicontinuous system mail fail to be transitive (as shown in Example \ref{ex4}). Further, as the above arguments hold good when the element of the orbit is replaced by element of the orbital hull, a similar set of arguments establish equivalence of denseness of orbital hulls (among elements of the space $X$). Also, for autonomous systems, it is known that a transitive system with dense set of periodic points is necessarily sensitive. However, as the governing rule may vary significantly in a non-autonomous setting, an analogous version of the result fails to hold good in a non-autonomous setting. However, as totally transitive systems guarantee visiting times to intersect multiples of each integer, totally transitive systems with dense set of transitive points are necessarily sensitive. We now establish our claims below.
\end{remark}
	
\begin{result}
For any equicontinuous non-autonomous system $(X,\mathbb{F})$ generated by a commutative family of homeomorphisms, if $\mathcal{O}_H(x)$ is dense (in $X$) for some $x\in X$ then $\mathcal{O}_H(x)$ is dense (in $X$) for all $x\in X$
\end{result}
	
\begin{proof}
The proof follows from discussions in Remark \ref{tth}.
\end{proof}

\begin{example}
Let $S^1$ be the unit circle and let the sequence $(f_n)$ be defined as $f_n(\theta) = \left\{
	\begin{array}{lr}
		{\theta+ 2\pi \sum\limits_{i=1}^k \frac{1}{i}} &  :n=2k-1, \\
		{\theta- 2\pi \sum\limits_{i=1}^k \frac{1}{i}} & : n=2k. \\		
	\end{array}
	\right.$

Firstly, note that as $\sum\limits_{i=1}^{\infty} \frac{1}{i}=\infty$, any point traverses the circle infinitely often (and hence passes across the origin infinitely often). Also, at the end of $n=2k+1$ iterations, any point $\theta$ rotates effectively by an angle ($2\pi \sum\limits_{i=1}^{k+1} \frac{1}{i}) (mod 2\pi)$. As rotations after $n$ iterations are of magnitude less than $\frac{1}{n}$, orbit of any point is $\frac{1}{n}$-dense in $X$ (for any $n\in\N$) and hence dense in  $X$ and thus the system is transitive. Also, as $\omega_{2n}(x)=x$ for all $n\in\Z$, every point in $S^1$ is periodic (of period $2$). However, as the maps involved are isometries, the system is equicontinuous. Thus, transitive system with dense set of periodic points need not exhibit sensitive dependence on initial conditions.
\end{example}

\begin{result}\label{tth}
For any non-autonomous system $(X,\mathbb{F})$ generated by a commutative family of homeomorphisms, if $(X,\mathbb{F})$ is totally transitive with dense set of periodic points then $(X,\mathbb{F})$ is sensitive.
\end{result}
	
\begin{proof}
Let $(X,\mathbb{F})$ be totally transitive with dense set of periodic points, $x\in X$ and $U$ be any $\frac{1}{n}$-neighborhood of $X$. As set of periodic points is dense, $U$ contains a periodic point $p$ (say of order $r$). Also, if $diam(X)>k$ there exists $y\in X$ such that $d(x,y)>\frac{k}{2}$. As $(X,\mathbb{F})$ is $r$-transitive, there exists $u\in U,~~n_y\in\Z$ such that $d(\omega_{rn_y}(u),y)<\frac{1}{n}$. Also as $\omega_{rn_y}(p)=p$, for a sufficiently large $n$, we have $d(\omega_{rn_y}(u),\omega_{rn_y}(p))>\frac{k}{4}$. Consequently, neighborhood of any point in $X$ expands to diameter greater than $\frac{k}{4}$ and hence  $(X,\mathbb{F})$ is sensitive.
\end{proof}

\begin{result}
For any minimal non-autonomous system $(X,\mathbb{F})$ generated by a commutative family of homeomorphisms, $(X,\mathbb{F})$ is either equicontinuous or exhibits a dense set of sensitive points.
\end{result}
	
\begin{proof}
Let $(X,\mathbb{F})$ be minimal system generated by a commutative family of homeomorphisms. If $(X,\mathbb{F})$ is equicontinuous then the result holds trivially. If not, let $x$ be a point of sensitivity (with sensitivity constant $\eta$) and let $k\in \Z$ be fixed. Let $\epsilon>0$ be given and let $U=S(\omega_k(x),\epsilon)$.  As $\omega_k$ is continuous, there exists $\delta>0$ such that $\omega_k(S(x,\delta))\subset S(\omega_k(x),\epsilon)$. As each $\omega_k$ is a homeomorphism, there exists $\eta'>0$ such that $d(a,b)\geq\eta$ implies $d(\omega_k(a),\omega_k(b))\geq\eta'$. Further, as $x$ is a point of sensitivity, there exists $y\in S(x,\delta)$ and $r\in\Z$ such that $d(\omega_r(x),\omega_r(y))\geq\eta$. Thus, we have  $d(\omega_k(\omega_r(x)),\omega_k(\omega_r(y)))\geq\eta'$ and hence $d(\omega_r(\omega_k(x)),\omega_r(\omega_k(y)))>\eta'$ (as $\mathbb{F}$ is commutative). As $\omega_k(y)\in U$, $\omega_k(x)$ is a point of sensitivity and hence sensitivity of a point ensures sensitivity at each element of the orbit. As the arguments can repeated for elements of the orbit, sensitivity at a point $x$ ensures sensitivity at elements of the orbital hull. Finally, as $X$ is minimal, orbital hull of $x$ is dense in $X$ and the proof is complete.
\end{proof}

\section*{Acknowledgement}
The first author thanks MHRD for financial support. The second author thanks National Board for Higher Mathematics (NBHM) for financial support.

\bibliography{xbib}

\end{document}